\documentclass[reqno, 12pt]{amsart}
\makeatletter
\let\origsection=\section \def\section{\@ifstar{\origsection*}{\mysection}} 
\def\mysection{\@startsection{section}{1}\z@{.7\linespacing\@plus\linespacing}{.5\linespacing}{\normalfont\scshape\centering\S}}
\makeatother        

\usepackage{amsmath,amssymb,amsthm}
\usepackage{mathrsfs}
\usepackage{mathabx}\changenotsign

\usepackage{xcolor}
\usepackage[backref]{hyperref}
\hypersetup{
    colorlinks,
    linkcolor={red!60!black},
    citecolor={green!60!black},
    urlcolor={blue!60!black}
}

\usepackage[abbrev,msc-links,backrefs]{amsrefs}
\usepackage[T1]{fontenc}
\usepackage{lmodern}

\usepackage[english]{babel}

\usepackage{geometry}
\usepackage{enumitem}
\def\rmlabel{\upshape({\itshape \roman*\,})}

\let\polishlcross=\l
\def\l{\ifmmode\ell\else\polishlcross\fi}

\let\emptyset=\varnothing
\let\setminus=\smallsetminus

\makeatletter
\def\moverlay{\mathpalette\mov@rlay}
\def\mov@rlay#1#2{\leavevmode\vtop{   \baselineskip\z@skip \lineskiplimit-\maxdimen
   \ialign{\hfil$\m@th#1##$\hfil\cr#2\crcr}}}
\newcommand{\charfusion}[3][\mathord]{
    #1{\ifx#1\mathop\vphantom{#2}\fi
        \mathpalette\mov@rlay{#2\cr#3}
      }
    \ifx#1\mathop\expandafter\displaylimits\fi}
\makeatother

\newtheorem{theorem}             {Theorem}[section]
\newtheorem{lemma}     	[theorem] {Lemma}        
   
\newtheorem{property}  	[theorem] {Property}   
\newtheorem{definition}	[theorem] {Definition}

\newtheorem{fact}	[theorem] {Fact}

\let\:=\colon

\def\calb{{\mathcal B}}

\def\calq{{\mathcal Q}}

\def\calt{{\mathcal T}}

\let\epsilon\varepsilon
\let\eps\varepsilon

\def\DISC{\mathop{\text{\rm DISC}}\nolimits}

\def\BDD{\mathop{\text{\rm BDD}}\nolimits}
\def\TUPLE{\mathop{\text{\rm TUPLE}}\nolimits}

\def\PAIR{\mathop{\text{\rm PAIR}}\nolimits}

\begin{document}

\title{Counting results for sparse pseudorandom hypergraphs~II}

\author[Y.~Kohayakawa]{Yoshiharu Kohayakawa}

\author[G.~O.~Mota]{Guilherme Oliveira Mota}
\address{Instituto de Matem\'atica e Estat\'{\i}stica, Universidade de S\~ao Paulo, S\~ao Paulo, Brazil}
\email{\{\,yoshi\,|\,mota\,\}@ime.usp.br}

\author[M.~Schacht]{Mathias Schacht}
\address{Fachbereich Mathematik, Universit\"at Hamburg, Hamburg, Germany}
\email{schacht@math.uni-hamburg.de}

\author[A.~Taraz]{Anusch Taraz}
\address{Institut f\"ur Mathematik, Technische Universit\"at Hamburg--Harburg, Hamburg, Germany}
\email{taraz@tuhh.de}

\thanks{Y.~Kohayakawa was partially supported by FAPESP (2013/03447-6, 2013/07699-0), 
	CNPq (310974/2013-5, 459335/2014-6), NUMEC/USP (Project MaCLinC/USP) and the 
	NSF (DMS~1102086). 
	G.~O.~Mota was supported by FAPESP (2009/06294-0, 2013/11431-2, 2013/20733-2). 
	M.~Schacht was supported by the Heisenberg-Programme of the DFG (grant SCHA 1263/4-1). 
	A.~Taraz was supported in part by DFG grant TA 309/2-2.
	The cooperation was supported by a joint CAPES/DAAD PROBRAL (333/09 and 430/15).}


\begin{abstract}
We present a variant of a universality result of R\"odl~[On universality of graphs with uniformly distributed edges, Discrete
Math. 59 (1986), no. 1-2, 125--134]
for sparse, $3$-uniform hypergraphs contained in strongly jumbled 
hypergraphs. 
One of the ingredients of our proof is a counting lemma for fixed 
hypergraphs in sparse ``pseudorandom'' uniform hypergraphs, which is proved in the companion paper [Counting results 
for sparse 
pseudorandom hypergraphs I].
\end{abstract}

\maketitle

\section{Introduction}
We say that a graph $G=(V,E)$ satisfies property $\calq(\eta,\delta,\alpha)$ if, for 
every subgraph $G[S]$ induced by 
$S\subset V$ with $|S|\geq\eta |V|$, we have
$(\alpha - \delta){|S|\choose 2} < |E(G[S])| < (\alpha + \delta){|S|\choose 2}$.
In~\cites{KoRo03,Ro86}, answering affirmatively a question posed by Erd\H{o}s (see, 
e.g.,\cite{Er79}~and~\cite{Bo04}*{p.~363}; see also~\cite{Ni01}), R\"odl 
proved the following result.

\begin{theorem}\label{thm:rodl2}
For all $k\geq 1$ and $0<\alpha$, $\eta<1$, there exist $\delta$, $n_0>0$ such that 
the following holds for all integer $n\geq n_0$. 

Every $n$-vertex graph $G$ that satisfies $\calq(\eta,\delta,\alpha)$ contains all graphs with $k$ 
vertices as induced subgraphs.
\end{theorem}

The quantification in Theorem~\ref{thm:rodl2} is what makes it unexpected. Indeed, note that $\eta$ is not required to be small, 
it is allowed to be any constant less than~$1$.

We prove a variant of this 
result, which allows one to count the
number of copies (not necessarily induced) of certain fixed $3$-uniform 
linear hypergraphs 
in spanning subgraphs of sparse ``jumbled'' $3$-uniform hypergraphs.

The concept of jumbledness~\cites{Th87a,Th87b} is well-known for graphs (see also~\cites{ChGr02, ChGr08, ChGrWi88, 
KrSu06}).
Let $\Gamma=(V,E)$ be a $3$-uniform hypergraph and let $X\subset{V\choose 2}$ and $Y\subset V$ be given. 
Denote by $E_{\Gamma}(X,Y)$ the set of triples in $\Gamma$ containing a pair in $X$ and a vertex in 
$Y$. Write $e_{\Gamma}(X,Y)$ for $|E_{\Gamma}(X,Y)|$. We say that 
$\Gamma$ is  $(p,\beta)$\emph{-jumbled} if, for all subsets $X\subset {V\choose 2}$ and 
$Y\subset V$, we have $\big|e_{\Gamma}(X,Y) - p|X||Y|\big|\leq \beta\sqrt{|X||Y|}$.
A hypergraph $H$ is called \emph{linear} if every pair of edges shares at 
most one vertex.
An edge $e$ of a linear $\ell$-uniform hypergraph~$H$ is a \emph{connector} if 
there exist $v\in 
V(H)\setminus e$ and $\ell$ edges $e_1,\ldots,e_\ell$ containing $v$ such that $|e\cap e_i|=1$ for 
$1\leq i\leq \ell$. 
Note 
that, for $\ell=2$, a connector is an edge that is contained in a triangle.

We prove a result that allows us to count the number of copies of 
small linear, connector-free $3$-uniform 
hypergraphs $H$ contained in certain $n$-vertex $3$-uniform spanning subhypergraphs $G_n$ of 
$(p,o(p^2 n^{3/2}))$-jumbled 
hypergraphs, for sufficiently large $p$ and $n$. We remark that, 
if $p\gg n^{-1/4}$, then the random $3$-uniform hypergraph, where each possible 
edge exists with probability $p$ independently of all other edges, is $(p,\gamma p^2 
n^{3/2})$-jumbled with high 
probability for all 
$\gamma>0$. 
Therefore, our result applies to dense enough random $3$-uniform hypergraphs.

This paper is organized as follows. In Section~\ref{sec:main} we state the main result of this paper 
(Theorem~\ref{thm:ext-rodl2}) and we discuss the structure of its proof. 
Section~\ref{sec:lemmas} contains the statements and the proofs of the lemmas involved in the proof of 
Theorem~\ref{thm:ext-rodl2}. Section~\ref{sec:proof} contains the proof of Theorem~~\ref{thm:ext-rodl2}. The final 
section contains some concluding remarks.

\section{Main result}\label{sec:main}

We start by generalizing property 
$\calq(\eta,\delta,\alpha)$ to $3$-uniform hypergraphs. We say that a $3$-uniform hypergraph 
$G=(V,E)$ satisfies  
property $\calq'(\eta,\delta,q)$ if, for all $X\subset{V\choose 2}$ and $Y\subset V$ with $|X|\geq 
\eta{|V|\choose 
2}$ and $|Y|\geq \eta |V|$, we have $(1-\delta)q|X||Y|\leq |E_G(X,Y)|\leq (1+\delta)q|X||Y|$. Considering the 
cardinality of $E_G(X,Y)$ for certain $X\subset{V\choose 2}$ and $Y\subset V$ to obtain information on the subhypergraphs of $G$ 
has recently been shown to be fruitful (see~\cites{ReRoSc16+b,ReRoSc16+a}).

Given a pair $\{v_1,v_2\}\in{V\choose 2}$, define $N_G(\{v_1,v_2\})=\big\{v_3\in 
V\colon \{v_1,v_2,v_3\}\in E\}$.
We say that a $3$-graph $G=(V,E)$ satisfies property $\BDD(k,C,q)$ if, for all 
$1\leq r\leq k$ and for all distinct $S_1,\ldots,S_r\in {V\choose {2}}$, we have 
$|N_G(S_1)\cap\ldots\cap N_G(S_r)| \leq Cnq^r$.

An \emph{embedding} of a hypergraph $H$ into another hypergraph $G$ is an injective mapping $\phi\colon 
V(H)\to V(G)$ such that $\{\phi(v_1),\ldots,\phi(v_k)\}\in E(G)$ whenever $\{v_1,\ldots,v_k\}\in E(H)$. We denote by 
$\mathcal{E}(H,G)$ the family of embeddings from $H$ into $G$. The following variant of Theorem~\ref{thm:rodl2} for $3$-uniform 
hypergraphs is our main result. 

\begin{theorem}\label{thm:ext-rodl2}
For all $0<\eps$, $\alpha$, $\eta<1$, $C>1$, and integer $k\geq 4$, there exist 
$\delta,\gamma>0$ such that if $p=p(n)\gg n^{-1/k}$ and $p=p(n)=o(1)$ and $n$ is 
sufficiently large, then the 
following holds for every $\alpha p\leq q\leq p$ and every $\beta\leq\gamma p^2n^{3/2}$.
Suppose that
\begin{enumerate}[label=\rmlabel]
 \item $\Gamma=(V,E_\Gamma)$ is an $n$-vertex $(p,\beta)$-jumbled $3$-uniform hypergraph;
 \item $G=(V,E_G)$ is a spanning subhypergraph of $\Gamma$ with $|E_G|=q{n\choose 3}$ and $G$ satisfies 
	     $\calq'(\eta,\delta,q)$ and  $\BDD(k,C,q)$.
\end{enumerate}
Then for every linear $3$-uniform connector-free hypergraph $H$ on 
$k$ vertices we have
\begin{equation*}
\big||\mathcal{E}(H,G)| - n^{k}q^{|E(H)|}\big| < \eps n^{k}q^{|E(H)|}.
\end{equation*}
\end{theorem}

The proof of Theorem~\ref{thm:ext-rodl2} requires several techniques.
First, we shall prove that, under the conditions of the theorem, $G$ satisfies a 
strong property involving degrees and co-degrees 
(see~Lemmas~\ref{lemma:pair->tuple},~\ref{lemma:Q->disc}~and~\ref{lemma:disc->pair}).
After that we 
 use an embedding result (Lemma~\ref{lemma:counting}) proved in~\cite{KoMoScTa14+I}. Before we 
discuss the scheme of 
the proof, let us 
define some hypergraph properties, called \emph{Discrepancy}, \emph{Pair}, and \emph{Tuple}.

\begin{property}[$\DISC$ -- Discrepancy property]
Let $G=(V,E)$ be a $3$-uniform hypergraph and let $X$, $Y\subset V$ be given. We say that the pair $(X,Y)$ 
satisfies 
$\DISC(q,p,\eps)$ in $G$ $($or $(X,Y)_G$ satisfies $\DISC(q,p,\eps))$ if for all $X'\subset{X\choose 
2}$ 
and $Y'\subset Y$ we have
\begin{equation*}
 \big|e_G(X',Y')-q|X'||Y'|\big|\leq \eps p{|X|\choose 2}|Y|.
\end{equation*}
Furthermore, if $(V,V)$ satisfies $\DISC(q,p,\eps)$ in $G$, then we say that the hypergraph 
$G$ satisfies $\DISC(q,p,\eps)$. 
\end{property}

For a $3$-uniform hypergraph $G=(V,E)$, a set of vertices $Y\subset V$, and pairs 
$S_1,S_2\in {V\choose 2}$ we denote $N_G(S_1)\cap Y$ by $N_G(S_1;Y)$ and 
$N_G(S_1)\cap N_G(S_2)\cap Y$ by $N_G(S_1,S_2;Y)$.

\begin{property}[$\PAIR$ -- Pair property]
Let $G=(V,E)$ be a $3$-uniform hypergraph and let~$X$, $Y\subset V$ be given. We say that the pair~$(X,Y)$ 
satisfies 
$\PAIR(q,p,\delta)$ in $G$ $($or $(X,Y)_G$ satisfies $\PAIR(q,p,\delta))$ if the following conditions 
hold:
\begin{equation*}
 \begin{aligned}
 \sum_{S_1\in {X\choose 2}} \big||N_G(S_1;Y)|-q|Y|\big|&\leq \delta p {|X|\choose 2}|Y|,\\
 \sum_{S_1\in {X\choose 2}} \sum_{S_2\in {X\choose 2}}\big||N_G(S_1,S_2;Y)|-q^2|Y|\big|&\leq \delta 
p^2 {|X|\choose 2}^2|Y|.
 \end{aligned}
\end{equation*}
Furthermore, if $(V,V)$ satisfies $\PAIR(q,p,\delta)$ in $G$, then we say that the hypergraph 
$G$ satisfies $\PAIR(q,p,\delta)$.
\end{property}

\begin{property}[$\TUPLE$ -- Tuple property]
We define $\TUPLE(\delta,q)$ as the family of $n$-vertex $3$-uniform hypergraphs $G=(V,E)$ such 
that the following two conditions hold:
\begin{enumerate}[label=\rmlabel]
 \item $\big||N_G(S_1)|-nq\big|<\delta nq$ for all but at most $\delta{n\choose 2}$ sets 
$S_1\in {V\choose 2}$;
 \item $\big||N_G(S_1)\cap N_G(S_2)|-nq^2\big|<\delta nq^2$ for all but at most 
$\delta{{n\choose 2}\choose 2}$ 
pairs $\{S_1,S_2\}$ of distinct sets in ${V\choose 2}$.
\end{enumerate}
\end{property}

The next result allows us to obtain property $\TUPLE$ from $\PAIR$. Since its proof is simple we will omit it.

\begin{lemma}\label{lemma:pair->tuple}
For all $0<\alpha\leq 1$ and $0<\delta<1$ there exists $\delta'>0$ such that if a $3$-uniform 
hypergraph $G$ 
satisfies 
$\PAIR(q,p,\delta')$ for $\alpha p\leq q\leq p$, then $G$ satisfies $\TUPLE(\delta,q)$.
\end{lemma}

In what follows we explain the organization of the proof. Consider the setup of 
Theorem~\ref{thm:ext-rodl2}. 
In order to obtain the conclusion of the theorem, we will use 
a counting result (Lemma~\ref{lemma:counting}), which requires that $G$ satisfies properties 
$\BDD$ and $\TUPLE$ for the appropriate parameters. Since $G$ 
satisfies $\BDD$ by hypothesis, it suffices to prove that $G$ satisfies $\TUPLE$. Using 
Lemma~\ref{lemma:Q->disc} it is possible to obtain $\DISC$ from property $\mathcal{Q}'$. Then, 
using that $G$ satisfies $\DISC$ one can show that $G$ satisfies $\PAIR$ using 
Lemma~\ref{lemma:disc->pair}, which implies $\TUPLE$ by Lemma~\ref{lemma:pair->tuple}. The 
quantification used in these implications is carefully analyzed in Section~\ref{sec:proof}.

\section{Main lemmas}\label{sec:lemmas}

We start by stating the counting lemma needed in the proof of Theorem~\ref{thm:ext-rodl2}. In order 
to apply it to a $3$-uniform 
$n$-vertex hypergraph $G$, we shall prove that $G$ satisfies $\TUPLE(\delta,q)$ for 
a sufficiently small $\delta$ and 
sufficiently large $0<q=q(n)\leq 1$. Since Lemma~\ref{lemma:pair->tuple} allows us to obtain 
$\TUPLE$ from $\PAIR$, we need to proof that $G$ satisfies $\PAIR(q,p,\delta')$ for a sufficiently 
small $\delta'$ and appropriate functions $p$ and $q$. This is done using 
Lemmas~\ref{lemma:Q->disc}~and~\ref{lemma:disc->pair}, which are proved, respectively, in the 
Subsections~\ref{sec:q-disc}~and~\ref{sec:disc-pair}

Given a $3$-uniform hypergraph $H$, we define parameters $d_H=\max\{\delta(J)\colon J\subset H\}$ 
and~$D_H=\min\{3d_H,\Delta(H)\}$. The following result, proved in~\cite{KoMoScTa14+I}, is our counting 
lemma.

\begin{lemma}\label{lemma:counting}
Let $k\geq 4$ be an integer and let 
$\eps>0$, $C>1$ and an integer $d\geq 2$ be fixed. Let $H$ be a linear $3$-uniform connector-free 
hypergraph on 
$k$ vertices such that $D_H\leq d$. Then, there exists $\delta>0$ for which 
the following holds for any $q=q(n)$ with $q\gg n^{-1/d}$ and $q=o(1)$ and for sufficiently large~$n$. 

If $G$ is an $n$-vertex $3$-uniform hypergraph with $|E(G)|=q{n\choose 3}$ hyperedges and $G$ satisfies 
$\BDD(D_H,C,q)$ and $\TUPLE(\delta,q)$, then
\begin{equation*}
\big||\mathcal{E}(H,G)| - n^{k}q^{|E(H)|}\big| < \eps n^{k}q^{|E(H)|}.
\end{equation*}
\end{lemma}

\subsection{\texorpdfstring{$\calq'$}{Q'} implies \texorpdfstring{$\DISC$}{DISC}}\label{sec:q-disc}
Given a $3$-uniform hypergraph $G=(V,E)$ and subsets $A\subset {V\choose 2}$ and 
non-empty $B\subset V$, the \emph{$q$-density} between $A$ and $B$ is defined as 
\[
	d_q(A,B)=\frac{|E_G(A,B)|}{q|A||B|}\,.
\]
Before we state the main result of this subsection, Lemma~\ref{lemma:Q->disc}, we shall prove the 
following result, which is inspired by a result in~\cite{Ro86} for graphs.

\begin{lemma}\label{lemma:Rodl-extension}
For all $0<\eta<1$ and $0<\eps^*<(1-\eta)/3$, there exists $\delta>0$ such 
that, if $G=(V,E)$ is an $n$-vertex $3$-uniform hypergraph that satisfies $\calq'(\eta,\delta,q)$, 
then 
the 
following holds.

For every $C\subset {V\choose 2}$ and $D\subset V$ such that $|C|$ is a multiple of 
$\lceil\eps^*{n\choose 
2}\rceil$ and $|D|$ is a multiple of $\lceil\eps^* n\rceil$, we have 
\begin{equation*}
1-\eps^* < d_q(C,D) < 1+\eps^*.
\end{equation*}
\end{lemma}

\begin{proof}
Fix $\eta>0$ and $0<\eps^*<(1-\eta)/3$. Let $\delta={{\eps^*}^3}/24$ and put $t=1/\eps^*$. Suppose 
$G=(V,E)$ is an 
$n$-vertex $3$-uniform hypergraph that satisfies $\calq'(\eta,\delta,q)$.
Now, fix $C\subset {V\choose 2}$ and $D\subset V$ such that $|C|=k_1\lceil\eps^*{n\choose 
2}\rceil$ and $|D|=k_2\lceil\eps^* n\rceil$ for some positive integers $k_1$ and $k_2$. 
Let $C_1,\ldots, C_{k_1}$ and $D_1,\ldots, D_{k_2}$ be, respectively, partitions of $C$ 
and $D$ such that $|C_1|=\ldots=|C_{k_1}|=\lceil\eps^*{n\choose 2}\rceil$ and 
$|D_1|=\ldots=|D_{k_2}|=\lceil\eps^*n\rceil$. Now we partition the sets 
${V\choose 2}\setminus C$ and $V\setminus D$, 
respectively, in sets $C_{k_1+1},\ldots,C_t$ and $D_{k_2+1},\ldots,D_t$ such that 
$|C_{k_1+1}|=\ldots=|C_t|=\lceil\eps^*{n\choose 2}\rceil$ and 
$|D_{k_2+1}|=\ldots=|D_t|=\lceil\eps^*n\rceil$. Note 
that $|C_t|\leq \eps^*{n\choose 2}$ and $|D_t|\leq \eps^* n$.

We divide the rest of the proof into two parts. First, we 
prove that for any triple $i,j,j'\in[t-1]$, $|e(C_i,D_j) - e(C_i,D_{j'})| \leq 6\delta{n\choose 
2}nq$, and for any triple $i,i',j\in[t-1]$, $|e(C_i,D_j) - e(C_{i'},D_j)| \leq 6\delta{n\choose 
2}nq$. To finish the 
proof we put these estimates together to show that $1-\eps^* < d_q(C,D) < 1+\eps^*$.

Put $X=C_2\cup\ldots\cup C_t$ and $Y=D_3\cup\ldots\cup D_t$. Since $\eps^*<(1-\eta)/3$, we have 
$|X|=(t-2)\lceil\eps^*{n\choose 2}\rceil +|C_t|\geq (t-2)\eps^*{n\choose 2}\geq\eta {n\choose 2}$ 
and  
$|Y|=(t-3)\lceil\eps^*n\rceil + |D_t|\geq (t-3)\eps^*n\geq 
\eta n$. Therefore, using $\calq'(\eta,\delta,q)$, the following two inequalities hold.
\begin{equation}\label{eq:D1D2}
\big|e(X,D_1\cup Y) - e(X,D_2\cup Y)\big|\leq 2\delta |X|(|D_1|+|Y|)q,
\end{equation}

\begin{equation}\label{eq:C1Dj}
\left|\frac{e(C_1\cup X, Y)}{(|C_1|+|X|)|Y|q} - 
\frac{e(C_1\cup X,D_j\cup Y)}{(|C_1|+|X|)(|D_j|+|Y|)q}\right|\leq 2\delta,\text{ for }j\in\{1,2\}.
\end{equation}
Now we define the following for $j\in\{1,2\}$
\begin{equation*}
p_{1j}=\frac{e(C_1\cup X, Y)}{(|C_1|+|X|)|Y|q} - \frac{e(C_1\cup X, 
Y)+e(X,D_j)}{(|C_1|+|X|)(|D_j|+|Y|)q}.
\end{equation*}
By \eqref{eq:C1Dj}, the following holds for $j\in\{1,2\}$
\begin{equation}\label{eq:eC1Dj}
p_{1j}-2\delta \leq \frac{e(C_1,D_j)}{(|C_1|+|X|)(|D_j|+|Y|)q} \leq 
p_{1j}+2\delta.
\end{equation}
Note that $\big| e(X,D_1)-e(X,D_2) \big|=\big|e(X,D_1\cup Y) - e(X,D_2\cup 
Y)\big|$. Thus, using~\eqref{eq:D1D2}, we obtain the following inequality.
\begin{align}\label{eq:difference-p1j}
\left|p_{11}-p_{12}\right|=\left|\frac{e(X,D_1)-e(X,D_2)}{(|C_1|+|X|)(|D_1|+|Y|)q}
\right|\leq \left(\frac{|X|}{|C_1|+|X|}\right)2\delta<2\delta.
\end{align}
Putting \eqref{eq:eC1Dj} and \eqref{eq:difference-p1j} together, we obtain the following inequality.
\[
|e(C_1,D_1) - e(C_1,D_2)|
	<6\delta (|C_1|+|X|)(|D_1|+|Y|)q
	<6\delta{n\choose 2}nq\,.
\]
Applying the same strategy one can prove that, for any triple $i,j,j'\in[t-1]$,
\begin{equation}\label{eq:finalC}
|e(C_i,D_j) - e(C_i,D_{j'})| < 6\delta{n\choose 2}nq.
\end{equation}
Analogously, we obtain the following equation for any triple $i,i',j\in[t-1]$.
\begin{equation}\label{eq:finalD}
|e(C_i,D_j) - e(C_{i'},D_j)| < 6\delta{n\choose 2}nq.
\end{equation}
By~\eqref{eq:finalC}~and~\eqref{eq:finalD}, we have $|e(C_i,D_j) - e(C_{i'},D_{j'})| < 
12\delta{n\choose 2}nq$ for any $i,i',j,j'\in[t-1]$.
Therefore, 
\begin{align}\label{eq:density-CD}
|d_q(C_i,D_j) - d_q(C_{i'},D_{j'})| < \frac{12\delta{n\choose 2}nq}{|C_i||D_j|q}
			<\frac{12\delta}{(\eps^*)^2}
			=\frac{\eps^*}{2}
\end{align}
holds for any $i,i',j,j'\in[t-1]$.
Put $W_C=C_1\cup\ldots\cup C_{t-1}$ and $W_D=D_1\cup\ldots\cup D_{t-1}$. Since $|W_C|\geq \eta 
{n\choose 2}$ and 
$|W_D|\geq \eta n$, we know, by property $\calq'(\eta,\delta,q)$, that
\begin{equation}\label{eq:densidadeW}
1-\delta<d_q(W_C,W_D)<1+\delta.
\end{equation}

Suppose for a contradiction that there exist indexes $i_0,j_0\in[t-1]$ such that either 
$d_q(C_{i_0},D_{j_0})>1+\eps^*$ or 
$d_q(C_{i_0},D_{j_0})<1-\eps^*$. Then, by \eqref{eq:density-CD}, either for all $i,j\in[t-1]$ we 
have 
$d_q(C_{i},D_{j})>1+\eps^*/2$ or for all $i,j\in[t-1]$ we have $d_q(C_{i},D_{j})<1-\eps^*/2$. But 
note that
\begin{equation*}
d_q(W_C,W_D)=
\frac{\sum_{i,j\in [t-1]}d_q(C_i,D_j)|C_i||D_j|q}{|W_C||W_D|q}.
\end{equation*}
Then, either
\begin{align*}
d_q(W_C,W_D)&<\frac{(t-1)^2(1-\eps^*/2)\left\lceil\eps^*{n\choose 
2}\right\rceil\lceil\eps^* 
n\rceil q}{|W_C||W_D|q}= (1-\eps^*/2) < 1-\delta,
\end{align*}
or
\begin{align*}
d_q(W_C,W_D)&>\frac{(t-1)^2(1+\eps^*/2)\left\lceil\eps^*{n\choose 
2}\right\rceil\lceil\eps^* n\rceil q}{|W_C||W_D|q}= (1+\eps^*/2)> 1+\delta,
\end{align*}
a contradiction with \eqref{eq:densidadeW}. Therefore, for all $i,j\in[t-1]$,
\begin{equation}\label{eq:final-density}
1-\eps^*<d_q(C_i,D_j)<1+\eps^*.
\end{equation}
It remains to estimate the densities $d_q(C_{k_1},D_j)$ and $d_q(C_i,D_{k_2})$ with $k_1=t$ and 
$k_2=t$ for all $1\leq i\leq k_1$ and $1\leq j\leq k_2$. Note that $k_1=t$ ($k_2=t$) if and only if 
$\lceil\eps^*{n\choose 2}\rceil=\eps^*{n\choose 2}$ ($\lceil\eps^*{n}\rceil=\eps^*{n}$), but in 
these cases one can prove in the same way we proved~\eqref{eq:final-density}. 
Therefore, putting all these estimates together, we obtain $1-\eps^*<d_q(C,D)<1+\eps^*$.
\end{proof}

The next lemma shows how one can obtain discrepancy properties from $\calq'$ in spanning 
subhypergraphs of sufficiently jumbled $3$-uniform hypergraphs.

\begin{lemma}\label{lemma:Q->disc}
For all $0<\eps',\eta,\sigma<1$ there exists $\delta>0$ such that for every $\alpha>0$ there 
exists $\gamma>0$ such that the following 
holds.

Let $\Gamma=(V,E_\Gamma)$ be an $n$-vertex $(p,\beta)$-jumbled $3$-uniform hypergraph for  
$0<p=p(n)\leq 1$ 
such that $\alpha 
p\leq q\leq p$ and $\beta\leq \gamma pn^{3/2}$. Let $G=(V,E_G)$ be a spanning subhypergraph of 
$\Gamma$. If $G$ satisfies $\calq'(\eta,\delta,q)$, then every pair $(X,Y)_G$ with $X,Y\subset V$
such that $|X|,|Y|\geq \sigma n$ satisfies $\DISC(q,p,\eps')$.
\end{lemma}

\begin{proof}
Fix  $\eps',\eta,\sigma>0$ and let $\eps^*=\min\big\{{\eps'}^2\sigma^2/24,(1-\eta)/4\big\}$. Let 
$\delta'$ be the 
constant
given by Lemma~\ref{lemma:Rodl-extension} applied with $\eta$ and $\eps^*$. Put 
$\delta=\min\{\delta',\eps'\}$, 
$\alpha>0$ and $\gamma=\sigma^{3/2}\alpha\eps'/2$.

Suppose that $\alpha p\leq q\leq p$ and  $\beta\leq p\gamma n^{3/2}$.
Let $\Gamma=(V,E_{\Gamma})$ be an $n$-vertex $(p,\beta)$-jumbled $3$-uniform hypergraph and 
let $G=(V,E_G)$ be a spanning subhypergraph of $\Gamma$ such that~$G$ satisfies 
$\calq'(\eta,\delta,q)$. 
Let $(X,Y)_G$ be a 
pair with $X,Y\subset V$ such that $|X|,|Y|\geq \sigma n$. We want to prove that $(X,Y)_G$ 
satisfies 
$\DISC(q,p,\eps')$. For this, fix arbitrary subsets $X'\subset {X\choose 2}$ and $Y'\subset Y$. 
We will prove that 
$\big|e_G(X',Y')-q|X'||Y'|\big|\leq \eps'p{|X|\choose 2}|Y|$.\\

\noindent\textbf{Upper bound}.
First, consider the case where $|X'|\leq \eps'{|X|\choose 2}$ or $|Y'|\leq \eps' |Y|$. Note 
that, from the 
choice of $\gamma$ and $\beta$, since $|X|,|Y|\geq \sigma n$, we have
\begin{equation}\label{eq:betabound}
\beta\sqrt{|X'||Y'|} \leq \alpha\eps'p{|X|\choose 2}|Y|.
\end{equation}
Therefore, 
\begin{align}\label{eq:superior-parte1}
e_G(X',Y')	&\leq p|X'||Y'| + \beta\sqrt{|X'||Y'|}\nonumber\\
		&\leq q|X'||Y'| + (1-\alpha)p|X'||Y'| + \beta\sqrt{|X'||Y'|}\nonumber \\
		&\leq q|X'||Y'| + (1-\alpha)p\eps'{|X|\choose 2}|Y| + 
\beta\sqrt{|X'||Y'|}\nonumber\\
		&\leq q|X'||Y'| + \eps'p{|X|\choose 2}|Y|,
\end{align}
where the first inequality follows from the jumbledness of $\Gamma$ and the fact that $G$ is a 
subhypergraph of 
$\Gamma$, the second one follows from the value of $q$, the third one follows from the fact that 
 $|X'|\leq \eps'{|X|\choose 2}$ or $|Y'|\leq \eps' |Y|$, and the last one is a consequence 
of~\eqref{eq:betabound}. 
Thus, we may assume $|X'|> \eps'{|X|\choose 2}$ and $|Y'|> \eps' |Y|$. We 
consider four cases, 
depending on the size of $|X'|$ and $|Y'|$.\\

\noindent\textbf{Case 1: ($|X'|\geq (1-\eps^*){n\choose 2}$ and $|Y'|\geq (1-\eps^*)n$).}
By the choice of $\eps^*$, we have $|X'|\geq \eta{n\choose 2}$ and $|Y'|\geq \eta n$. By 
$\calq(\eta,\delta,q)$ we conclude that
\[
	e_G(X',Y')	\leq (1+\delta)q|X'||Y'|
		\leq q|X'||Y'| + \eps'p{|X|\choose 2}|Y|.
\]

\noindent\textbf{Case 2: ($|X'|< (1-\eps^*){n\choose 2}$ and $|Y'|< (1-\eps^*)n$).}
Note that, since $|X'|< (1-\eps^*){n\choose 2}$ and $|Y'|< (1-\eps^*)n$, there 
exist subsets 
$X^*\subset{V\choose 2}$ and $Y^*\subset V$ such that $X^*=X'\cup X''$ and $Y^*=Y'\cup Y''$, 
with $X'\cap 
X''=\emptyset$ and $Y'\cap Y''=\emptyset$, where $|X''|\leq\eps^*{n\choose 2}$ and $|X^*|$ is 
multiple 
of $\lceil 
\eps^*{n\choose 2}\rceil$, and $|Y''|\leq \eps^* n$ and $|Y^*|$ is a multiple of $\lceil 
\eps^*n\rceil$. Then, we can use 
Lemma~\ref{lemma:Rodl-extension} to obtain the following inequality.
\begin{align*}
e_G(X',Y')&\leq e_G(X^*,Y^*)
	  \leq (1+\eps^*)|X^*||Y^*|q\\
	  &\leq (1+\eps^*)q|X'||Y'| + 2q\big(|X'||Y''|+|X''||Y'|+|X''||Y''|\big).
\end{align*}
Since $\eps^*\leq \eps'^2\sigma^2/16$, we have $|X''|\leq \eps^*{n\choose 2}\leq (\eps'/8)|X'|$ and 
$|Y''|\leq \eps^* 
n\leq (\eps'/8)|Y'|$. Therefore,
\begin{align*}
e_G(X',Y')	&\leq (1+\eps^*)q|X'||Y'|  + 2q 
			\big(3(\eps'/8)|X'||Y'|\big)\\
		&\leq q|X'||Y'| + \frac{\eps'}{4}q|X'||Y'| + \frac{3\eps'}{4}q|X'||Y'|
		\leq q|X'||Y'| + \eps'p{|X|\choose 2}|Y|.
\end{align*}

\noindent\textbf{Case 3: ($|X'|\geq (1-\eps^*){n\choose 2}$ and $|Y'|< (1-\eps^*)n$).} 
As noticed before, since $|Y'|< (1-\eps^*)n$, there exist subsets $Y^*, Y''\subset V$ such that
$Y^*=Y'\cup Y''$ with $Y'\cap Y''=\emptyset$, where $|Y''|\leq \eps^* n$ and $|Y^*|$ is a multiple of 
$\lceil \eps^*n\rceil$. Note that there exist subsets $\tilde{X},X''\subset {V\choose 2}$ such 
that
$X'=\tilde{X}\cup X''$ with $\tilde{X}\cap X''=\emptyset$, where $|X''|\leq \eps^* {n\choose 2}$ 
and 
$|\tilde{X}|$ is a multiple of $\lceil \eps^*{n\choose 2}\rceil$.

If $X''$ is empty, then put $W''=\emptyset$. If $X''$ is not empty, then we ``complete'' 
$X''$ with elements of ${V\choose 2}$ to obtain $W''$ such that $X''\subset W''$ and 
$|W''|=\lceil 
\eps^*{n\choose 2} \rceil$ (note that possibly $W''\cap \tilde{X}\neq\emptyset$). Thus, $|\tilde{X}|+|W''|\leq 
|X'|+\eps^*{n\choose 2}$. By using
Lemma~\ref{lemma:Rodl-extension}, we have
\begin{align*}
e_G(X',Y')&\leq e_G(W'',Y^*)+e_G(\tilde{X},Y^*)\\
	&\leq (1+\eps^*)q\left(|Y^*||W''|+|Y^*||\tilde{X}|\right)\\
		&= (1+\eps^*)q\left(|Y'||W''|+|Y''||W''|+|Y'||\tilde{X}|+|Y''||\tilde{X}|\right)\\
		&\leq (1+\eps^*)q\left(|Y'|\left(|X'|+\eps^*{n\choose 2}\right) + 
|Y''|\left(|X'|+\eps^*{n\choose 
			2}\right)\right)\\
		&\leq (1+\eps^*)q|X'||Y'| + 2q\left(\eps^*{n\choose 2}|Y'| + |X'|\eps^* n + 
\eps^*{n\choose 2} \eps^* n\right).
\end{align*}
Since $\eps^*\leq \eps'^2\sigma^2/16$, we have $\eps^*{n\choose 2}\leq (\eps'/8)|X'|$ and 
$\eps^* n\leq (\eps'/8)|Y'|$. Therefore,
\begin{align*}
e_G(X',Y')	&\leq q|X'||Y'| + \frac{\eps'}{4} q|X'||Y'| + 2q\left(\frac{3\eps'}{8}|X'||Y'|\right)\\
		&\leq q|X'||Y'| + \eps'{|X|\choose 2}|Y|p.
\end{align*}

\noindent\textbf{Case 4: ($|X'|< (1-\eps^*){n\choose 2}$ and $|Y'|\geq (1-\eps^*)n$).}
This case is analogous to Case 3.
\medskip

\noindent\textbf{Lower bound}.
If $|X'|\leq \eps'{|X|\choose 2}$ or 
$|Y'|\leq \eps' |Y|$, then there is nothing to prove, because $\eps'{|X|\choose 
2}|Y|p>q|X'||Y'|$. Therefore, 
assume that $|X'|> \eps'{|X|\choose 2}$ and $|Y'|> \eps' |Y|$. Clearly, there exist subsets 
$\tilde{X}\subset{V\choose 2}$ and $\tilde{Y}\subset V$ such that $X'=\tilde{X}\cup X''$ and 
$Y'=\tilde{Y}\cup 
Y''$, with $\tilde{X}\cap X''=\emptyset$ and $\tilde{Y}\cap Y''=\emptyset$, where 
$|X''|\leq\eps^*{n\choose 2}$ 
and $|\tilde{X}|$ is a multiple of $\lceil 
\eps^*{n\choose 2}\rceil$ and $|Y''|\leq \eps^* n$ and $|\tilde{Y}|$ is a multiple of $\lceil 
\eps^*n\rceil$. 

Since $\eps^*\leq \eps'^2\sigma^2/8$, we have 
\[
	|X''|\leq \eps^*{n\choose 2}\leq (\eps'/4)|X'|\leq \big(\eps'/4(1-\eps^*)\big)|X'|
\] 
and $|Y''|\leq 
\eps^* n\leq 
(\eps'/4(1-\eps^*))|Y'|$. Then, by 
Lemma~\ref{lemma:Rodl-extension}, since $e_G(X',Y')\geq e_G(\tilde{X},\tilde{Y})$, we have
\begin{align*}
e_G(X',Y')	&\geq (1-\eps^*)|\tilde{X}||\tilde{Y}|q\\
		&=(1-\eps^*)q\big(|X'||Y'|-|X'||Y''|-|X''||Y'|+|X''||Y''|\big)\\
		&\geq (1-\eps^*)q|X'||Y'| - (1-\eps^*)q\big(|X'||Y''|+|X''||Y'|\big)\\
		&\geq q|X'||Y'| - \eps^*q|X'||Y'| - (1-\eps^*)q 
			\big((\eps'/2(1-\eps^*))|X'||Y'|\big)\\
		&\geq q|X'||Y'| - \frac{\eps'}{2}q|X'||Y'| - \frac{\eps'}{2}q|X'||Y'|\\
		&\geq q|X'||Y'| - \eps'{|X|\choose 2}|Y|p.
\end{align*}
\end{proof}

\subsection{\texorpdfstring{$\DISC$}{DISC} implies \texorpdfstring{$\PAIR$}{PAIR}}\label{sec:disc-pair}
The next lemma, which is a variation of Lemma 9 in \cite{KoRoScSk10}, makes it possible to 
obtain $\PAIR$ from $\DISC$ in spanning subhypergraphs of sufficiently jumbled 
$3$-uniform hypergraphs.

\begin{lemma}\label{lemma:disc->pair}
For all $0<\alpha\leq 1$ and $\delta'>0$ there exists $\eps'>0$ such that for all $\sigma>0$ there 
exist 
$\gamma>0$ such that the following holds for sufficiently large $n$.

Suppose that
\begin{enumerate}[label=\rmlabel]
 \item $\Gamma=(V,E_\Gamma)$ is an $n$-vertex $3$-uniform $(p,\beta)$-jumbled hypergraph with 
$p\geq 1/\sqrt{n}$,
 \item $G=(V,E_G)$ is a spanning subhypergraph of $\Gamma$, and
 \item $X,Y\subset V$ with $|X|,|Y|\geq \sigma n$.
\end{enumerate}
Then, the following holds.
If $\beta\leq \gamma p^2n^{3/2}$ and $(X,Y)_G$ satisfies $\DISC(q,p,\eps')$ for some $q$ 
with $\alpha p\leq q\leq p$, then $(X,Y)_G$ satisfies $\PAIR(q,p,\delta')$.
\end{lemma}

We need the following results in order to prove Lemma~\ref{lemma:disc->pair}. First, consider the 
following fact, which 
is similar to \cite{KoRoScSk10}*{Fact~13}.

\begin{fact}\label{af:Fato13}
Let $\Gamma$ be a $3$-uniform $(p,\beta)$-jumbled hypergraph. Let $U\subset {V\choose 2}$ 
and $W\subset 
V$ and $\xi>0$. If we have $|N_{\Gamma}(\{x,y\},W)|\geq (1+\xi)p|W|$ for every $\{x,y\}\in U$
or we have $|N_{\Gamma}(\{x,y\},W)|\leq (1-\xi)p|W|$ for every $\{x,y\}\in U$, then
\begin{equation*}
|U||W|\leq\frac{\beta^2}{\xi^2p^2}.
\end{equation*}
\end{fact}

\begin{proof}
Let $\Gamma$, $U$, $W$ and $\xi$ be as in the statement and suppose that for every $\{x,y\}\in U$ we have
$\big|N_{\Gamma}(\{x,y\},W)\big|\geq (1+\xi)p|W|$. Suppose for a
contradiction 
that $|U||W|>\frac{\beta^2}{\xi^2p^2}$. Then,
$e_{\Gamma}\left(U,W\right) \geq |U|(1+\xi)p|W| >p|U||W| + \beta\sqrt{|U||W|}$,
a contradiction to the jumbledness of $\Gamma$. The case where $|N_{\Gamma}(\{x,y\},W)|\leq 
(1-\xi)p|W|$ for 
every $\{x,y\}\in U$ is analogous.
\end{proof}

Our next result, Lemma~\ref{lemma:lemma21} below, is very similar to~\cite{KoRoScSk10}*{Lemma~21}, but in 
Lemma~\ref{lemma:lemma21} we consider bipartite 
graphs $\Gamma=\big({V\choose 2},V;E_\Gamma\big)$ instead of $\Gamma=(U,V;E_{\Gamma})$ 
in~\cite{KoRoScSk10}, 
and we consider 
subsets $X_1,X_2$ of ${V\choose 2}$ with $|X_1|,|X_2|\geq \eta {n\choose 2}$ instead of subsets 
$X_1,X_2$ of 
$V$ with $|X_1|,|X_2|\geq \eta n$. Due to this fact, the value of $\beta$ in 
Lemma~\ref{lemma:lemma21} is $\gamma 
p^2 n^{3/2}$, while in \cite{KoRoScSk10}*{Lemma~21} we have $\beta=\gamma pn$. The proof of 
Lemma~\ref{lemma:lemma21} is identical to the proof of~\cite{KoRoScSk10}*{Lemma~21} and we 
omit it here.

Let $\Gamma=(V,E_\Gamma)$ be a graph and let $X$, $Y\subset V$. As usual, we denote by $e_{\Gamma}(X,Y)$ the 
number of edges of $\Gamma$ with one end-vertex in $X$ and one end-vertex in $Y$, where edges 
contained in $X\cap Y$ are counted twice. We need to define jumbledness and discrepancy for graphs.

\begin{definition}[Jumbledness for graphs]
 We say that $\Gamma=(V,E_\Gamma)$ is a $(p,\beta)$\emph{-jumbled graph} if, for all subsets $X$, 
$Y\subset V$, we have 
$\big|e_{\Gamma}(X,Y) - p|X||Y|\big|\leq 
\beta\sqrt{|X||Y|}$.
Furthermore, a bipartite graph $\Gamma_B=(U,V;E)$ is called  
$(p,\beta)$\emph{-jumbled} if, for all $X\subset U$ and 
$Y\subset V$, we have $|e_{\Gamma}(X,Y) - p|X||Y|\big|\leq \beta\sqrt{|X||Y|}$.
\end{definition}

\begin{property}[Discrepancy for graphs]
Let $G=(V,E)$ be a graph and let $X$, $Y\subset V$ be disjoint. We say that $(X,Y)$ 
satisfies 
$\DISC(q,p,\eps)$ in $G$ (or $(X,Y)_G$ satisfies $\DISC(q,p,\eps)$) if for all $X'\subset X$ and $Y'\subset Y$ we have
\begin{equation*}
 \big|e_G(X',Y')-q|X'||Y'|\big|\leq \eps p|X||Y|.
\end{equation*}
\end{property}

\begin{lemma}\label{lemma:lemma21}
For all positive real $\varrho_0$ and $\nu$, there exists a positive real $\mu$ such that, for all 
$\sigma'>0$, there exist
$\gamma>0$ and $n_0>0$ such that for all $n\geq n_0$, the following holds.

Suppose
\begin{enumerate}[label=\rmlabel]
 \item $\Gamma=\big({V\choose 2},V;E_\Gamma\big)$ is a bipartite $(p,\beta)$-jumbled graph 
with 
$|V|\geq n$, $p\geq 1/\sqrt{n}$ and $\beta\leq\gamma p^2n^{3/2}$,
 \item $X_1,X_2\subset {V\choose 2}$ and $Y\subset V$ with $|X_1|,|X_2|\geq \sigma'{n\choose 
2}$, $|Y|\geq 
\sigma' n$,
 \item $B=(X_1,X_2;E_B)$ is an arbitrary bipartite graph.
\end{enumerate}
Then, if $(X_1,X_2)_B$ satisfies $\DISC(\varrho,1,\mu)$ for some $\varrho$ with
$\varrho_0\leq \varrho\leq 1$, then for all but at most $\nu |Y|$ vertices $y\in Y$, the pair
$\big(N_\Gamma(y,X_1),N_\Gamma(y,X_2)\big)_B$ satisfies $\DISC(\varrho,1,\nu)$.
\end{lemma}

We need two facts before proving of Lemma~\ref{lemma:disc->pair}.

\begin{fact}[{\cite{KoRoScSk10}*{Fact~22}}]\label{af:Fato22}
Suppose $\varrho_0>0$, $\mu>0$ and $B=(X,E_B)$ is a graph with $|E_B|\geq \varrho_0{|X|\choose 2}$. 
Then there exist disjoint subsets $X_1,X_2\subset X$ such that
\begin{enumerate}[label=\rmlabel]
 \item $(X_1,X_2)_B$ satisfies $\DISC(\varrho,1,\mu)$ for some $\varrho\geq \varrho_0$,
 \item $|X_1|,|X_2|\geq \zeta |X|$ for $\zeta={\varrho_0}^{100/\mu^2}/4$.
\end{enumerate}
\end{fact}

\begin{fact}\label{af:hiperJumbledtoGraphJumbled}
Let $\Gamma=(V,E)$ be a $3$-uniform hypergraph and let $\Gamma'=\big({V\choose 2},V;E'\big)$ be a 
bipartite 
graph, where $E'=\big\{\{\{v_1,v_2\},v\}\colon \{v_1,v_2\}\in{V\choose
2}, v\in V$ and $\{v_1,v_2,v\}\in E\big\}$. Then,~$\Gamma$ is $(p,\beta)$-jumbled if and only 
if 
$\Gamma'$ is $(p,\beta)$-jumbled.
\end{fact}

We have stated all the tools needed in the proof of Lemma~\ref{lemma:disc->pair}. This proof is very similar to the proof 
of~\cite{KoRoScSk10}*{Lemma~9}.

\begin{proof}[Proof of Lemma~\ref{lemma:disc->pair}]
Let $0<\alpha\leq 1$ and $0<\delta'<1$ be given. Put $\xi=\delta'/6$, $\varrho_0=\delta'/50$ and 
$\nu=\alpha^2\xi\varrho_0/64$. Let $\mu$ be obtained by an application of Lemma~\ref{lemma:lemma21} 
with 
parameters~$\varrho_0$ and $\nu$. Without loss of generality, assume $\mu<\xi\varrho_0/4$. 
Let $\zeta=\varrho_0^{100/\mu^2}/4$ be given and put 
$\eps'=\min\big\{\alpha\delta'^2/36,(\alpha^3\xi\varrho_0\zeta/64)^2\big\}$. Now 
fix  
$\sigma>0$ and let $\sigma'=\zeta\sigma^2/2$. Following the quantification of 
Lemma~\ref{lemma:lemma21} applied with 
parameter $\sigma'$ we obtain 
$\gamma'$ and $n_0$. 
Then, put  
\begin{equation*}
\gamma=\min\big\{\gamma',\sqrt{\sigma^3\delta'/12},(\alpha/2)\sqrt{\xi\varrho_0\sigma\sigma'/24}\big\}.
\end{equation*}
Finally, consider $n$ sufficiently large and suppose $p\geq 1/\sqrt{n}$.

Fix $\beta\leq \gamma p^2n^{3/2}$ and consider a $3$-uniform $(p,\beta)$-jumbled hypergraph 
$\Gamma=(V,E_{\Gamma})$ 
such that $|V|= n$ and let $G=(V,E_G)$ be a spanning subhypergraph of $\Gamma$.
Let $X,Y$ be subsets of $V$ such that $|X|,|Y|\geq \sigma n$. Suppose that $(X,Y)_G$ satisfies
$\DISC(q,p,\eps')$ for some $q$ with $\alpha p\leq q\leq p$, i.e.,
for all $X'\subset{X\choose 2}$ and $Y'\subset Y$ the following holds.
\begin{equation}\label{eq:disc}
 \big|e_G(X',Y')-q|X'||Y'|\big|\leq \eps' p{|X|\choose 2}|Y|.
\end{equation}
We want to prove that the following inequalities hold:
\begin{align}
\sum_{S_1\in {X\choose 2}} \big||N_G(S_1;Y)|-q|Y|\big|&\leq \delta' p {|X|\choose 
2}|Y|,\label{eq:Par1}\\
\sum_{S_1\in {X\choose 2}} \sum_{S_2\in {X\choose 2}}\big||N_G(S_1,S_2;Y)|-q^2|Y|\big|&\leq \delta' 
p^2 {|X|\choose 
2}^2|Y|.\label{eq:Par2}
\end{align}
We start by verifying \eqref{eq:Par1}. For at most 
$\delta'{|X|\choose 
2}/6$ pairs $S\in {X\choose 2}$, we have 
\[
	\big||N_G(S,Y)|-q|Y|\big|>(\delta'/3) q|Y|\,.
\]
Indeed, otherwise there would be a set $B_X\subset {X\choose 2}$ with at least 
$\delta'{|X|\choose 2}/12$ elements such that, for all $\{x,x'\}\in B_X$, either
$|N_G(\{x,x'\},Y)|>(1+\delta'/3)q|Y|$ or for all of them we have $|N_G(\{x,x'\},Y)|<(1-\delta'/3)q|Y|$. In either case, we 
would have
\begin{align*}
\big|e_G(B_X,Y)-q|B_X||Y|\big| &> \frac{\delta'^2}{36}q{|X|\choose 2}|Y|
\geq \frac{\delta'^2\alpha}{36}p{|X|\choose 2}|Y|
\geq \eps' p{|X|\choose 2}|Y|,
\end{align*}
where the last inequality follows from the choice of $\eps'$. But this contradicts~\eqref{eq:disc} 
when we put $X'=B_X$ 
and $Y'=Y$.

Let $W$ be the set of pairs $S\in{X\choose 2}$ such that $|N_{\Gamma}(S,Y)|\geq 2p|Y|$.
By Fact~\ref{af:Fato13} applied to $W$ and $Y$ with $\xi=1$, we know that there exist at most
$\beta^2/p^2|Y|$ elements $S\in W$ such that $|N_{\Gamma}(S,Y)|\geq 2p|Y|$.
Therefore,
\begin{align*}
\sum_{S\in {X\choose 2}} \big||N_G(S,Y)|-q|Y|\big|&\leq {|X|\choose 2}\frac{\delta'}{3}q|Y| + 
\left(\frac{\delta'}{6}{|X|\choose 2}\right)2p|Y| + 
(\beta^2/p^2|Y|)|Y|\\
&\leq p {|X|\choose 2}|Y|\left(\frac{2\delta'}{3}\right) + (\beta/p)^2
\leq \delta' p {|X|\choose 2}|Y|,
\end{align*}
where the last inequality follows from the facts that $\beta\leq\gamma p^2 n^{3/2}$ 
and $\gamma\leq\sqrt{\sigma^3\delta'/12}$. We just 
proved that~\eqref{eq:Par1} holds. 

Suppose for a contradiction that~\eqref{eq:Par2} does not holds. Then, 
\begin{equation}\label{eq:contradicaoPar}
\sum_{S_1\in {X\choose 2}} \sum_{S_2\in {X\choose 2}}\big||N_G(S_1,S_2;Y)|-q^2|Y|\big|> \delta' p^2 
{|X|\choose 2}^2|Y|.
\end{equation}
Define the following sets of ``bad'' pairs.
\begin{align*}
\calb_1&= \left\{(S_1,S_2)\in {X\choose 2}\times{X\choose 2}\colon |N_{\Gamma}(S_1,Y)| > 
2p|Y|\right\},\\
\calb_2&= \left\{(S_1,S_2)\in{X\choose 2}\times{X\choose 2}\setminus\calb_1\colon 
|N_{\Gamma}(S_1,S_2,Y)| > 
4p^2|Y|\right\}.
\end{align*}
Since $\Gamma$ is $(p,\beta)$-jumbled, it follows that
\begin{align*}
|\calb_1|  \leq \frac{\beta^2}{p^2|Y|}{|X|\choose 2}
\leq \frac{\gamma^2 n^3 p^2}{|Y|}{|X|\choose 2}
\leq \frac{\gamma^2 n^2p^2}{\sigma}{|X|\choose 2}
\leq \frac{\delta'}{3}p^2 {|X|\choose 2}^2.
\end{align*}
where the first inequality follows from Fact~\ref{af:Fato13} applied to the sets 
\[
	W=\{S_1\in{X\choose 2}\colon |N_{\Gamma}(S_1,Y)|\geq 2p|Y|\}
\]
and $Y$ with $\xi=1$. The second inequality follows from the 
choice of $\beta$, the 
third one follows from $|Y|\geq\sigma n$, and the last one 
holds because
$|X|\geq\sigma n$ and $\gamma\leq\sqrt{\sigma^3\delta'/12}$.

We want to bound $|\calb_2|$ from above. By definition, if a pair of vertices belongs to~$\calb_2$, 
then it does 
not belong to $\calb_1$. Then, consider a pair of vertices $S_1\in {X\choose 2}$ such that 
$|N_{\Gamma}(S_1,Y)|\leq 
2p|Y|$. Consider a set $Y'\subset Y$ of size exactly $2p|Y|$ that contains $N_{\Gamma}(S_1,Y)$. 
Applying 
Fact~\ref{af:Fato13} to the sets $\big\{S_2\in{X\choose 2}\colon |N_{\Gamma}(S_2,Y')|\geq 
2p|Y'|\big\}$ and 
$Y'$ with $\xi=1$, 
we 
conclude that there are at most $\beta^2/p^2|Y'|$ pairs $S_2\in {X\choose 2}$ such that 
$|N_{\Gamma}(S_1,S_2,Y)| > 
4p^2|Y|$. Therefore,
\begin{align*}
|\calb_2|  \leq {|X|\choose 2}\frac{\beta^2}{p^22p|Y|}
\leq {|X|\choose 2}\frac{\gamma^2 n^2p}{2\sigma}
\leq \frac{\delta'}{6}p {|X|\choose 2}^2,
\end{align*}
The summation below is over the pairs $(S_1,S_2)\in {X\choose 2}\times{X\choose 2}\setminus\calb_1\cup\calb_2$.
By~\eqref{eq:contradicaoPar} and the upper bounds on $\calb_1$ and $\calb_2$ we conclude that
\begin{equation}\label{eq:B1B2saoIrrelevantes}
\begin{aligned}
\sum \big||N_G(S_1,S_2;Y)|-q^2|Y|\big| &> 
\delta' p^2 {|X|\choose 2}^2|Y| - 
|\calb_1||Y| -
|\calb_2|2p|Y|\\
&\geq \delta' p^2 {|X|\choose 2}^2|Y| - \frac{2\delta'}{3}p^2 {|X|\choose 2}^2|Y|\\
&= \frac{\delta'}{3} p^2 {|X|\choose 2}^2|Y|.
\end{aligned}
\end{equation}
The contribution of the pairs $(S_1,S_2)\notin (\calb_1\cup\calb_2)$ with 
$\big||N_G(S_1,S_2;Y)|-q^2|Y|\big|\leq 
\delta' q^2|Y|/6$ to the sum in
\eqref{eq:B1B2saoIrrelevantes} is at most
\begin{align}\label{eq:limitaCima}
\frac{\delta'}{6}p^2{|X|\choose 2}^2|Y|.
\end{align}
Note that, by the definition of $\calb_2$, for all $(S_1,S_2)\notin \calb_1\cup\calb_2$, the following 
holds.
\begin{equation*}
\big||N_G(S_1,S_2;Y)| - q^2|Y|\big|\leq \max\big\{q^2|Y|, (4p^2-q^2)|Y|\big\}\leq 4p^2|Y|. 
\end{equation*}
Hence, by \eqref{eq:B1B2saoIrrelevantes} and \eqref{eq:limitaCima}, there exist at least 
$\delta'{|X|\choose 
2}^2/24$ pairs
$(S_1,S_2)\in
{X\choose 2}\times{X\choose 2}\setminus(\calb_1\cup\calb_2)$ such that
\begin{equation}\label{eq:limitaParesXi}
\big||N_G(S_1,S_2;Y)| - q^2|Y|\big|> \frac{\delta'}{6}q^2|Y|=\xi q^2 |Y|. 
\end{equation}

Now let us define two auxiliary graphs $B^+$ and $B^-$ with vertex-set ${X\choose 2}$ and edge-sets 
as follows.
\begin{align*}
E(B^+)&= \left\{\{S_1,S_2\}\in {{X\choose 2}\choose 2}\colon (1+\xi)q^2|Y|<|N_G(S_1,S_2;Y)|\leq 
4p^2|Y|\right\}\label{eq:arestasGrafoRuim+}\\
E(B^-)&= \left\{\{S_1,S_2\}\in {{X\choose 2}\choose 2}\colon 
|N_G(S_1,S_2;Y)|<(1-\xi)q^2|Y|\right\}\nonumber.
\end{align*}
Since there are at least $\delta'{|X|\choose 2}^2/24$ pairs $(S_1,S_2)\in
{X\choose 2}\times{X\choose 2}$ such that \eqref{eq:limitaParesXi} holds, we have
\begin{align*}
\max\{e(B^+),e(B^-)\}\geq \frac{{|X|\choose 2}^2\delta'/24}{4} - {|X|\choose 2}
\geq \varrho_0 {{|X|\choose 2}\choose 2},
\end{align*}
where in the first inequality the term ``$4$'' in the denominator comes from the fact that now we 
are counting 
unordered pairs and the edges belongs either to $E(B^+)$ or $E(B^-)$. Furthermore, we discount the 
pairs $\{S_1,S_1\}$.

Suppose without lost of generality that $e(B^+)\geq \varrho_0 {{|X|\choose 2}\choose 2}$.
Then, Fact~\ref{af:Fato22} implies that there 
exist subsets 
$X_1,X_2\subset {X\choose 2}$ with $|X_1|$, $|X_2|\geq \zeta 
{|X|\choose 2}$ such that $(X_1,X_2)_{B^+}$ satisfies $\DISC(\varrho,1,\mu)$ for some 
$\varrho\geq\varrho_0$.

Recall that $\Gamma=(V,E_{\Gamma})$ is a $3$-uniform $(p,\beta)$-jumbled hypergraph with $n$ vertices. By 
Fact~\ref{af:hiperJumbledtoGraphJumbled}, the bipartite graph $\Gamma'=({V\choose 2},V;E_{\Gamma'})$, where
\[
	E_{\Gamma'}=\Big\{\big\{\{v_1,v_2\},v\big\}\colon \{v_1,v_2\}\in\tbinom{V}{2}, v\in V\text{ and }\{v_1,v_2,v\}\in E_{\Gamma}\Big\}
\]
is a $(p,\beta)$-jumbled graph. Note that $X_1,X_2\subset {X\choose 
2}\subset{V\choose 2}$ with 
$|X_1|$, $|X_2|\geq \zeta {|X|\choose 2}\geq\zeta {\sigma n\choose 2}\geq 
(\zeta\sigma^2/2){n\choose 2}\geq\sigma'{n\choose 2}$. Therefore, the 
hypotheses of 
Lemma~\ref{lemma:lemma21} are satisfied. By Lemma~\ref{lemma:lemma21} we conclude that for all but 
at most $\nu|Y|$ vertices $y\in Y$, the pair $(N_{\Gamma'}(y,X_1),N_{\Gamma'}(y,X_2))_{B^+}\text{ satisfies 
}\DISC(\varrho,1,\nu)$, which 
implies the following statement for all but 
at most $\nu|Y|$ vertices $y\in Y$.
\begin{equation}\label{eq:muitosDeYcomDisc}
(N_{\Gamma}(y,X_1),N_{\Gamma}(y,X_2))_{B^+}\text{ satisfies }\DISC(\varrho,1,\nu).
\end{equation}

Now let us estimate the number of triplets $(S_1,S_2,y)$ in $X_1\times X_2\times Y$ such that 
$\{S_1,S_2\}$ is an 
edge of $B^+$ and the pairs $S_1$ and $S_2$ belong to the neighbourhood of $y$ in $G$. Formally, we 
define such 
triplets 
as follows.
\[
\calt=\{(S_1,S_2,y)\in X_1\times X_2\times Y\colon S_1\in N_G(y,X_1),
S_2\in N_G(y,X_2),
\{S_1,S_2\}\in E(B^+)\}.
\]
By the definition of $B^+$ we have
\begin{equation}\label{eq:limiteInferiorTriplas}
\begin{aligned}
|\calt|&>(1+\xi)q^2|Y|e_{B^+}(X_1,X_2)
\geq (1+\xi)q^2|Y|(\varrho-\mu)|X_1||X_2|\\
&> \left(1+\frac{\xi}{2}\right)\varrho q^2|X_1||X_2||Y|,
\end{aligned}
\end{equation}
where in the second inequality we used the fact that $(X_1,X_2)_{B^+}\in\DISC(\varrho,1,\mu)$ and the 
last one follows from the choice of $\mu$.

Now we will give an upper bound on $|\calt|$ that contradicts~\eqref{eq:limiteInferiorTriplas}. For 
that, we write
$|\calt| = \sum_{y\in Y}e_{B^+}\big(N_G(y,X_1),N_G(y,X_2)\big)$. Put
\begin{equation*}
Y'=\big\{y\in Y\colon d_\Gamma(y,X_i)\leq 2p|X_i|\text{ for both }i=1,2\big\}.
\end{equation*}
By~\eqref{eq:muitosDeYcomDisc}, for all but at 
most $\nu|Y|$ vertices $y\in Y'$ we have
\begin{align*}
e_{B^+}(N_G(y,X_1),N_G(y,X_2))&\leq \varrho|N_G(y,X_1)\big||N_G(y,X_2)| \\
			      &\ \ \ + \nu|N_\Gamma(y,X_1)||N_\Gamma(y,X_2)|\\
			      &\leq \varrho\ d_G(y,X_1)d_G(y,X_2) + 4\nu p^2|X_1||X_2|.
\end{align*}
The last inequality follows from the fact that $y\in Y'$. Now we will bound the terms related to 
vertices in 
$Y\setminus Y'$. By Fact~\ref{af:Fato13}, we have
$|Y\setminus Y'|\leq \beta^2\big(1/p^2|X_1| + 1/p^2|X_2|\big)$.
Then,
\begin{align*}
|\calt|\leq &\sum_{y\in Y'}\left(\varrho\ d_G(y,X_1)d_G(y,X_2) + 4\nu p^2|X_1||X_2|\right) + \nu|Y| 
4p^2|X_1||X_2|\\
&+ \left(\frac{\beta^2}{p^2|X_1|} + \frac{\beta^2}{p^2|X_2|}\right)|X_1||X_2|.
\end{align*}
The next inequality is obtained by putting the following facts together:
$\nu=\alpha^2\xi\varrho/64$, 
$\gamma\leq 
(\alpha/2)\sqrt{\xi\varrho_0\sigma\sigma'/24}$, $q\geq \alpha p$, 
$|X_1|,|X_2|\geq
\sigma'{n\choose 2}$ and $|Y|\geq\sigma n$.
\begin{equation}\label{eq:Tporcima}
|\calt|\leq \varrho\sum_{y\in Y'}\left(d_G(y,X_1)d_G(y,X_2)\right) + \frac{\xi}{4}\varrho q^2 
|X_1||X_2||Y|.
\end{equation}
Define $Y_i''=\big\{y\in Y\colon d_G(y,X_i)>(1+\sqrt{\eps'})q|X_i|\big\}$ for both $i=1,2$. Since 
$(X,Y)_G\in\DISC(q,p,\eps')$, it is not hard to see that $|Y_i''|\leq
\sqrt{\eps'}p{|X|\choose 2}|Y|/q|X_i|$ for both $i=1,2$. Since $|X_1|,|X_2|\geq \zeta{|X|\choose 2}$,  
$q\geq \alpha p$ and $\eps'\leq(\alpha^3\xi\varrho_0\zeta/64)^2$, the following holds for both $i=1,2$.
\begin{equation*}
|Y_i''|\leq \frac{\xi\varrho\alpha^2}{64}|Y|.
\end{equation*}
Note that
\begin{align*}
\sum_{y\in Y'}\left(d_G(y,X_1)d_G(y,X_2)\right) &= \sum_{y\in Y'\setminus(Y_1''\cup 
Y_2'')}\left(d_G(y,X_1)d_G(y,X_2)\right)\\
&\ \ \ + \sum_{y\in Y'\cap(Y_1''\cup 
Y_2'')}\left(d_G(y,X_1)d_G(y,X_2)\right)\\
&\leq |Y|(1+\sqrt{\eps'})^2q^2|X_1||X_2| + 
\frac{\xi\varrho\alpha^2}{32}|Y|\left(4p^2|X_1||X_2|\right).
\end{align*}
Therefore, since $\sqrt{\eps'}\leq \xi/24$, the above inequality together with \eqref{eq:Tporcima} 
implies
\begin{equation*}
|\calt|\leq \left(1+\frac{\xi}{2}\right)\varrho q^2|X_1||X_2||Y|,
\end{equation*}
a contradiction with~\eqref{eq:limiteInferiorTriplas}.
\end{proof}

\section{Proof of the main result}\label{sec:proof}

In this section we show how to combine the lemmas presented in Section~\ref{sec:lemmas} in order to 
prove 
Theorem~\ref{thm:ext-rodl2}.

\begin{proof}[Proof of Theorem~\ref{thm:ext-rodl2}]
Let $\eps$, $\alpha$, $\eta>0$, $C>1$ and $k\geq 4$ be given. Let $H_1,\ldots,H_r$ be all 
the $k$-vertex
$3$-uniform hypergraphs which are linear and connector-free. Applying Lemma~\ref{lemma:counting} 
with parameters $k$, 
$C$, $\eps$ and $d=k$ for $H_1,\ldots, H_r$, we obtain, respectively, constants 
$\delta_1,\ldots,\delta_r$. Now put $\delta_{\min}=\min\{\delta_1,\ldots,\delta_r\}$.
Let $\delta'$ be given by 
Lemma~\ref{lemma:pair->tuple} applied with $\alpha$ and $\delta_{\min}$. Let $\eps'$ be given by 
Lemma~\ref{lemma:disc->pair} 
applied with $\alpha$ and $\delta'$. Lemma~\ref{lemma:Q->disc} applied with $\eps'$, $\eta$ and 
$\sigma=1$ gives 
$\delta$. Following the quantification of Lemma~\ref{lemma:Q->disc} applied with $\alpha$ we obtain 
$\gamma_1$. 
Finally, following the quantification of Lemma~\ref{lemma:disc->pair} applied with $\sigma=1$ we 
obtain $\gamma_2$.

Put $\gamma=\min\{\gamma_1,\gamma_2\}$. Let ${p=p(n)=o(1)}$ with $p\gg n^{-1/k}$ and let $q=q(n)$ 
be such that $\alpha 
p\leq q\leq p$. In what follows we suppose that $n$ is sufficiently large.

Let $\Gamma=(V,E_{\Gamma})$ be an $n$-vertex $(p,\beta)$-jumbled $3$-uniform hypergraph and let $G$ 
be a spanning 
subhypergraph of $\Gamma$ with $|E(G)|=q{n\choose 3}$ such that $G$ satisfies 
$\calq'(\eta,\delta,q)$ and  
$\BDD(k,C,q)$. Suppose that $\beta\leq\gamma p^2n^{3/2}$. We want to prove that $G$ contains 
$(1\pm\eps)n^{k}q^{|E(H)|}$ copies of all 
linear $3$-uniform connector-free hypergraphs $H$ with $k$ vertices. 
By Lemma~\ref{lemma:Q->disc}, our hypergraph $G$ satisfies $\DISC(q,p,\eps')$. Now apply 
Lemmas~\ref{lemma:disc->pair}~and~\ref{lemma:pair->tuple} in succession to deduce that $G$ satisfies $\PAIR(q,p,\delta')$ and 
$\TUPLE(\delta,q)$.
Now let $H$ be any linear $3$-uniform connector-free hypergraphs $H$ with $k$ vertices.
Since $G$ satisfies $\TUPLE(\delta,q)$ and $\BDD(k,C,q)$, by Lemma~\ref{lemma:counting}, we conclude that  
\begin{equation*}
\big||\mathcal{E}(H,G)| - n^{k}q^{|E(H)|}\big| < \eps n^{k}q^{|E(H)|}.
\end{equation*}
\end{proof}

\section{Concluding Remarks}\label{sec:concluding}

Most of the definitions in this paper generalize naturally to $k$-uniform hypergraphs, for 
$k$ larger than $3$. Lemma~\ref{lemma:counting} holds 
for $k$-uniform hypergraphs for every $k\geq 2$ (for details, see~\cite{KoMoScTa14+I}).
It would be interesting to obtain a version of Theorem~\ref{thm:ext-rodl2} for $k$-uniform 
hypergraphs when $k>3$, but unfortunately such a generalization presents 
new difficulties and will be considered elsewhere.

\end{document}